\documentclass{amsart}
\title{On a strengthening of J\'onssonness for $\aleph_\omega$}
\author{Monroe Eskew}
\address{Universit\"at Wien \\
Institut f\"ur Mathematik \\
Kurt G\"odel Research Center \\
Augasse 2-6, UZA 1 - Building 2 \\
1090 Wien \\
AUSTRIA}

\date{}

\usepackage{color} 
\usepackage{amssymb} 
\usepackage{amsmath} 
\usepackage[utf8]{inputenc} 
\usepackage{enumerate}
\usepackage{amsthm}
\usepackage{cite}

\newtheorem{theorem}{Theorem}
\newtheorem{lemma}[theorem]{Lemma}
\newtheorem{proposition}[theorem]{Proposition}
\newtheorem{corollary}[theorem]{Corollary}
\newtheorem*{question}{Question}

\DeclareMathOperator{\ran}{ran}
\DeclareMathOperator{\ot}{ot}

\DeclareMathOperator{\cf}{cf}

\DeclareMathOperator{\id}{id}

\DeclareMathOperator{\hull}{Hull}
\DeclareMathOperator{\chang}{\twoheadrightarrow}

\newcommand{\la}{\langle}
\newcommand{\ra}{\rangle}

\date{}

\begin{document}
\maketitle

\begin{abstract}
We discuss a system of strengthenings of ``$\aleph_\omega$ is J\'onsson'' indexed by real numbers, and identify a strongest one.  We give a proof of a theorem of Silver and show that there is a barrier to weakening its hypothesis.
\end{abstract}

This note is meant to be self-contained and relies only on basic facts about cardinals, Skolem functions, and the nonstationary ideal.  For background, see \cite{MR2768692}.

An infinite cardinal $\kappa$ is called \emph{J\'onsson} if for every structure $\frak A$ on $\kappa$ in a countable language, there is an elementary $\frak B \prec \frak A$ of size $\kappa$ such that $\frak B \not= \frak A$.  $\aleph_0$ is not J\'onsson, since we may include the function $n \mapsto n-1$ for $n>0$ in the language of a structure on the natural numbers.  It is easy to see that if $\kappa$ is not J\'onsson, then $\kappa^+$ is not J\'onsson, so $\aleph_\omega$ is the least possible J\'onsson cardinal.  The consistency of $\aleph_\omega$ being J\'onsson, relative to large cardinals, is open.

An equivalent way of stating that $\kappa$ is J\'onsson is to say that the set $\{ X \subseteq \kappa : |X| = \kappa$ and $X \not= \kappa \}$ is stationary, which means that for every function $F : [\kappa]^{<\omega} \to \kappa$, there is $X$ in the above set that is closed under $F$.  J\'onsson cardinals are closely connected to Chang's Conjecture:

\begin{proposition}
Suppose $\kappa$ is J\'onsson.  Then there is a regular $\mu<\kappa$ such that the Chang's Conjecture $(\kappa,\mu)\chang(\kappa,{<}\mu)$ holds, which says that $\{ X \subseteq \kappa : |X| = \kappa$ and $|X \cap \mu| < \mu \}$ is stationary.
\end{proposition}

\begin{proof}
Let $F: [\kappa]^{<\omega} \to \kappa$ be a function.  Let $M \prec (H_{\kappa},\in,F)$ be of size $\kappa$ with $\kappa \subseteq M$.  Since $\kappa$ is J\'onsson, there is $N \prec M$ such that $| N \cap \kappa | = \kappa$ and $N \cap \kappa \not= \kappa$.  There must be a cardinal $\mu \in N \cap \kappa$ such that $N \cap \mu \not= \mu$.  This is clear if $\kappa$ is a limit cardinal.  If $\kappa = \rho^+$ and $\rho \subseteq N$, then for every $\alpha \in M \cap \kappa$ there is a surjection $f_\alpha : \rho \to \alpha$ in $M$, and thus $\alpha = f_\alpha[\rho] \subseteq N$.   Let $\mu$ be the least cardinal in $N$ such that $N \cap \mu \not= \mu$.  By the minimality of $\mu$, $N \cap \mu$ cannot be cofinal in $\mu$, since otherwise $N \cap \alpha = \alpha$ for all $\alpha \in \mu \cap N$, and thus $\mu \subseteq N$ .  If $\mu$ is singular, then $\cf(\mu) + 1 \subseteq N$, so $N$ is cofinal in $\mu$.  Note that $N \cap \kappa$ is closed under $F$.

For $X \subseteq \kappa$, let  $\mu_X \in X$ be the least cardinal regular cardinal such that $|X \cap \mu| < \mu$.  This is a regressive function defined on a stationary set, so it is constant on a stationary subset by Fodor's Theorem.  Thus for some $\mu < \kappa$, $(\kappa,\mu)\chang(\kappa,{<}\mu)$.
\end{proof}

We can generalize Chang's Conjecture as follows.  Given an index set $I$ and sequences of cardinals $\langle \lambda_i \rangle_{i \in I}$ and $\langle \kappa_i \rangle_{i \in I}$, the notation
$$ \langle \lambda_i \rangle_{i \in I} \twoheadrightarrow \langle \kappa_i \rangle_{i \in I} $$
stands for the following assertion:
Let $\lambda =  \sup_{i \in I} \lambda_i$.  For every $F : [\lambda]^{<\omega} \to \lambda$, there is $X \subseteq \lambda$ closed under $F$ such that $|X \cap \lambda_i |= \kappa_i$ for all $i$.  When only finitely many cardinals $\lambda_0 < ... < \lambda_n$ appear in the sequence, the assertion is conventionally denoted by
\[ (\lambda_n, ..., \lambda_0) \twoheadrightarrow (\kappa_n,...,\kappa_0). \]
Replacing any ``$\kappa_i$'' with ``${<}\kappa_i$'' in the above notation signifies that $| X \cap \lambda_i | < \kappa_i$.  We note some easy implications:

\begin{itemize}
\item If $J \subseteq I$ and $\langle \lambda_i \rangle_{i \in I} \twoheadrightarrow \langle \kappa_i \rangle_{i \in I}$, then $\langle \lambda_i \rangle_{i \in J} \twoheadrightarrow \langle \kappa_i \rangle_{i \in J}$.
\item If $\langle \lambda_i \rangle \twoheadrightarrow \langle \kappa_i \rangle$ and $\langle \kappa_i \rangle \twoheadrightarrow \langle \mu_i \rangle$, then $\langle \lambda_i \rangle \twoheadrightarrow \langle \mu_i \rangle$.
\end{itemize}

\begin{lemma}
\label{cardlem}
Suppose $\lambda$ is a cardinal.  There is a structure $\frak A$ on $\lambda$ in a finite language such that all $\frak B \prec \frak A$ have the property that $\{ \alpha < \lambda : \exists \beta \ot(\frak B \cap \alpha) = \aleph_\beta \}$ contains only cardinals.
\end{lemma}

\begin{proof}
Let $f : \lambda \to \lambda$ be defined by $f(\alpha) = |\alpha|$, and let $g : \lambda^2 \to \lambda$ be such that for all $\alpha < \lambda$, the function on $\alpha$ given by $\beta \mapsto g(\alpha,\beta)$ is an injection from $\alpha$ to $| \alpha |$.  Let $\frak A$ have $f$ and $g$ in its language.  Suppose $\frak B \prec \frak A$ and  $\aleph_\beta \leq | \frak B|$.  Let $\alpha \in \frak B$ be least such that $| \frak B \cap \alpha | = \aleph_\beta$.  If $\alpha$ is not a cardinal, then $f(\alpha) \in \frak B$ and $f(\alpha) < \alpha$. $g(\alpha,\cdot) : \frak B \cap \alpha \to \frak B \cap f(\alpha)$ is an injection, but $| \frak B \cap f(\alpha)| < \aleph_\beta$, a contradiction.
\end{proof}

%

\begin{lemma}[Folklore]
\label{sups}
For all $\lambda$, there is a structure $\frak A$ on $\lambda$ with a complete set of Skolem functions 
such that all $\frak B \prec \frak A$ have the following property: If $\mu < \kappa$ are in $\frak B$ and $\kappa$ is a regular cardinal, then $\sup(\hull^{\frak A}( \frak B \cup \mu ) \cap \kappa) = \sup(\frak B \cap \kappa)$.
\end{lemma}

\begin{proof}We may assume we have a complete set of Skolem functions $\langle f_i \rangle_{i < \omega}$ for $\frak A$ which is closed under compositions and has the property that for all $i < \omega$, there is $j < \omega$ such that
$$f_j : (\alpha_0,\alpha_1,\vec \beta) \mapsto  \sup \{ \gamma < \alpha_0 : (\exists \vec \alpha \in [\alpha_1]^{<\omega} ) \ f_i(\vec \alpha,\vec \beta) = \gamma \}.$$
Suppose $\gamma \in \hull^{\frak A}( \frak B \cup \mu) \cap \kappa$.  Then $\gamma = f_i(\vec \alpha, \vec \beta)$ for some $i<\omega$, $\vec \alpha \in [\mu]^{<\omega}$, $\vec \beta \in \frak B^{<\omega}$.  Then there is some $f_j$ as above, so that $\gamma \leq f_j(\kappa,\mu,\vec \beta) \in \frak B$.  
Since $\mu < \kappa$ and $\kappa$ is regular, $f_j(\kappa,\mu,\vec \beta) < \kappa$.  
Thus $\sup(\hull^{\frak A}( \frak B \cup \mu ) \cap \kappa) \leq \sup(\frak B \cap \kappa)$.
\end{proof}

Let $\theta \geq \aleph_\omega$ be a cardinal.  
If $M \prec H_\theta$, then $M$ satisfies the conclusion of Lemma \ref{cardlem} and contains functions witnessing that each $\aleph_n$ is not J\'onsson.  If $|M \cap \aleph_\omega | = \aleph_\omega$, let $\chi_M : \omega \to \omega$ be such that $\chi_M(n)$ is the $m \geq n$ such that  $\aleph_n = \ot(M \cap \aleph_m)$.  Note that $\chi_M(0) = 0$ always, and if $\aleph_\omega \nsubseteq M$, then for all but finitely many $n$, $\chi_M(n) > n$.

Let us call a function $f$ \emph{increasing} if $n<m$ implies $f(n) < f(m)$, and \emph{monotone} if $n < m$ implies $f(n) \leq f(m)$.  Let us say that $\aleph_\omega$ is \emph{$f$-J\'onsson} if the collection $\{ M \prec H_{\aleph_\omega} : | M |= \aleph_\omega $ and $\chi_M = f \}$ is stationary.  Clearly, if $f \not= \id$ and $\aleph_\omega$ is $f$-J\'onsson, then $\aleph_\omega$ is J\'onsson.

The following argument is due to Jack Silver.  It has appeared in other guises in \cite{MR3803069} and \cite{MR520190}.

\begin{theorem}[Silver]
If $\aleph_\omega$ is J\'onsson and $2^{\aleph_0} < \aleph_\omega$, then $\aleph_\omega$ is $f$-J\'onsson for some increasing $f : \omega \to \omega$, $f\not=
id$.  Therefore, there is an increasing sequence of natural numbers $\la n_i \ra_{i<\omega}$ such that $\la \aleph_{n_{i+1}} \ra_{i<\omega} \chang \la \aleph_{n_i} \ra_{i<\omega}$.
\end{theorem}

\begin{proof}
Suppose $\aleph_\omega$ is J\'{o}nsson and $2^{\aleph_0} = \aleph_{m_0}$.  Let $F : \aleph_\omega^{<\omega} \to \aleph_\omega$ and let $\lhd$ be a well-order of $H_{\aleph_\omega}$.  Let $M \prec (H_{\aleph_\omega},\in,\lhd,F)$ be of size $\aleph_\omega$ and such that $\aleph_\omega \nsubseteq M$.  
Let $N = \hull(M \cup \aleph_{m_0})$.  
We claim that $\chi_M(n) = \chi_N(n)$ for $n>m_0$.
  For if not, let $n>m_0$ be least such that $m= \chi_N(n) < \chi_M(n)$.  Then $\ot(N \cap \aleph_m) = \aleph_n$ and $\cf(M \cap \aleph_m) < \aleph_n$.  But by Lemma \ref{sups}, $\sup(M \cap \aleph_m) = \sup(N \cap \aleph_m)$, a contradiction.  Since $2^{\aleph_0} \subseteq N$, $\chi_N \in N$, and $N$ is closed under $F$.  Thus the following collection is stationary:
$$\{ N \prec H_{\aleph_\omega} : |N| = \aleph_\omega, \aleph_\omega \nsubseteq N, \text{ and } \chi_N \in N \}.$$
By Fodor's Theorem, there is a stationary $S$ contained in the above set and a function $f$ such that $\chi_N = f$ for all $N \in S$.

For the last claim, let $n_0$ be the largest $n$ such that $f(n) = n$, and let $n_1$ be such that $n_0 < n_1 < f(n_0 + 1)$.  For $i \geq 1$, if $n_i$ is given, let $n_{i+1} = f(n_i)$.  If $M \in S$, then $| M \cap \aleph_{n_1} | = \aleph_{n_0}$, and for $i \geq 1$, $| M \cap \aleph_{n_{i+1}} | = |M \cap \aleph_{f(n_i)}| = \aleph_{n_i}$.
\end{proof}

If $\aleph_\omega$ is $f$-J\'onsson, for what other functions $g$ is it $g$-J\'onsson?  We have seen above that we can alter $f$ by replacing an initial segment of it with the identity function:  If $\aleph_\omega$ is $f$-J\'onsson, then it is $(\id \restriction n \cup f \restriction (\omega \setminus n))$-J\'onsson for all $n$.  What about adjustments on an infinite set?


For functions in $\omega^\omega$, we define exponentiation of one function by another as a generalization of composition.  As usual, for positive integers $n$, $f^n = f \circ \dots \circ f$, where we compose $f$ with itself $n$ times.  We let $f^0 = \id$.  For functions $f,g \in \omega^\omega$, we define $f^g$ by $f^g(n) =  f^{g(n)}(n)$.

\begin{proposition}
Suppose $f,g,h \in\omega^\omega$, $f$ is increasing, and $g,h$ are monotone.
\begin{enumerate}
\item $f(n) \geq n$ for all $n$.
\item $f(n)-n$ is monotone.
\item $f^g$ is increasing.
\item $(f^g)^h = f^k$ for some monotone function $k$.
\end{enumerate}
\end{proposition}

\begin{proof}
\hspace{1mm}
\begin{enumerate}
\item Suppose $f(n) \geq n$.  Then $f(n+1) \geq f(n)+1 \geq n+1$.
\item $f(n+1) - (n+1) \geq (f(n)+1) - (n+1) = f(n)-n$.
\item Suppose $n_0<n_1$.  Since $f$ is increasing, we inductively see that $f^m(n_0) < f^m(n_1)$ for all $m$.  Since $g$ is monotone, $f^g(n_0) = f^{g(n_0)}(n_0) < f^{g(n_0)}(n_1) \leq f^{g(n_1)}(n_1) = f^g(n_1)$.

\item We show by induction on $m$ that for all $n,m$, there is a number $k_m(n)$ such that $(f^g)^m(n) = f^{k_m(n)}(n)$, and that if either $n$ or $m$ is held fixed, then $k_m(n)$ is monotone function of the other variable.  Our desired function $k$ is then defined by $k(n) = k_{h(n)}(n)$.  For the base case $m = 0$,  this is true because $(f^g)^0 = \id$, so $k_0$ is the constant function with value 0.  Suppose this is true for $m$.  For each $n$,
\begin{align*}
(f^g)^{m+1} (n)	&= f^g \circ (f^g)^m(n) \\
			&= f^g \circ f^{k_m(n)}(n) \\
			&= \underbrace{f \circ \dots \circ f}_{g(f^{k_m(n)}(n)) \text{ times}} ( f^{k_m(n)}(n)).
\end{align*}
Thus $k_{m+1}(n) = k_m(n) + g(f^{k_m(n)}(n)) \geq k_m(n)$.  To show that $k_{m+1}$ is a monotone function of $n$, suppose $n_0 < n_1$.
\begin{align*}
k_{m+1}(n_0)	&= k_m(n_0) + g(f^{k_m(n_0)}(n_0)) \\
			&\leq k_m(n_1) + g(f^{k_m(n_0)}(n_0)) \\
			&\leq k_m(n_1) + g(f^{k_m(n_0)}(n_1)) \\
			&\leq k_m(n_1) + g(f^{k_m(n_1)}(n_1)) = k_{m+1}(n_1).
\end{align*}
\end{enumerate} 
\end{proof}

%

More generally, if $\vec f = \la f_i : i < \omega \ra \subseteq \omega^\omega$ and $g \in \omega^\omega$, we define $\vec f^g$ by 
$$\vec f^g(n) = f_0 \circ f_1 \circ \dots \circ f_{g(n)-1}(n).$$
For any $f,g \in \omega^\omega$, if we put $\vec f = \la f,f,f,\dots \ra$, then $f^g = \vec f^g$.  If $f,g \in \omega^\omega$, then $f \circ g$ is obtained by putting $\vec h = \la f,g,g,g,\dots \ra$ and $k = \la 2,2,2,\dots \ra$ and taking $\vec h^k$.

\begin{proposition}
Suppose $\vec f = \la f_i : i < \omega \ra \subseteq \omega^\omega$, each $f_i$ is increasing, and $g \in \omega^\omega$ is monotone.  Then $\vec f^g$ is increasing.
\end{proposition}

\begin{proof}
Suppose $n_0<n_1$.  For any increasing functions $h,k$, we have $h(n_0)<h(k(n_1))$.  Since the composition of increasing functions is increasing, for any $m_0\leq m_1$, we have $f_{0} \circ \dots \circ f_{m_0}(n_0) < f_{0} \circ \dots \circ f_{m_1} (n_1)$.
\end{proof}

\begin{theorem}
Suppose $\vec f = \la f_i : i < \omega \ra$ is a sequence of functions such that $\aleph_\omega$ is $f_i$-J\'onsson for each $i$, and $g \in \omega^\omega$ is monotone.  Then $\aleph_\omega$ is $\vec f^g$-J\'onsson.  Let $s_1$ be the function that sends $0 \mapsto 0$ and $n \mapsto n+1$ for positive $n$.  If  $\aleph_\omega$ is $s_1$-J\'onsson, then it is $f$-J\'onsson for all increasing $f \in \omega^\omega$ such that $f(0) = 0$.
\end{theorem}

\begin{proof}
Let $F : \aleph_\omega^{<\omega} \to \aleph_\omega$, let $\theta \geq \aleph_\omega$, and let $M \prec (H_\theta,\in,F)$ be of size $\aleph_\omega$ with $\aleph_\omega \subseteq M$.  We may assume that $g$ is not the constant function with value 0, since in that case the conclusion holds trivially.  If $n_0$ is the least $n$ such that $g(n)>0$, then defining $g'$ such that $g'(n) = 1$ for $n < n_0$ and $g'(n) = g(n)$ for $n \geq n_0$, and putting $f'_i = \id \restriction n_0 \cup f_i \restriction (\omega \setminus n_0)$ for each $i$, we have that $(\vec f')^{g'} = \vec f^g$.  By Lemma \ref{sups}, $\aleph_\omega$ is $f'_i$-J\'onsson for each $i$.  Thus it suffices assume $g$ takes only positive values.

Let $N_{0,0} \prec M$ be such that $\chi_{N_{0,0}} = f_0$.  If $\pi : N_{0,0} \to \bar N_{0,0}$ is the transitive collapse, then we can take $N_{0,1}^* \prec \bar N_{0,0}$ such that $\chi_{N_{0,1}^*} = f_1$.  Then setting $N_{0,1} = \pi^{-1}[N_{0,1}^*]$, we have that $\chi_{N_{0,1}} = f_0 \circ f_1$.  Continuing in this way up to $g(1)-1$, we have a model $N_0 = N_{0,g(1)-1} \prec M$ such that $\chi_{N_0} = f_0 \circ \dots \circ f_{g(1)-1}$, so $\chi_{N_0}(1)= \vec f^g(1)$.  Since $f_i(0) = 0$ for all $i$, $\chi_{N_0}(0) = \vec f^g(0)$.

Let $\pi : N_0 \to \bar N_0$ be the transitive collapse.  Since $\aleph_\omega$ is $f_{g(1)}$-J\'onsson, it is $(\id \restriction 2 \cup f_{g(1)} \restriction (\omega \setminus 2))$-J\'onsson by Lemma \ref{sups}.  Thus we can take $N_{1,g(1)}^* \prec \bar N_0$ witnessing this.  We continue as in previous paragraph up to $g(2) -1$, obtaining $N^*_{1,g(2)-1} \prec \bar N_0$ such that $\chi_{N^*_{1,g(2)-1}} = \id \restriction 2 \cup (f_{g(1)} \circ \dots \circ f_{g(2) - 1}) \restriction (\omega \setminus 2)$.  Then we let $N_1 = \pi^{-1}[N^*_{1,g(2)-1}] \prec N_0$, and we have that for all $i > 1$,
$$\chi_{N_1}(i) = f_0 \circ \dots \circ f_{g(1)-1} \circ  f_{g(1)} \circ \dots \circ  f_{g(2)-1}(i).$$  
We also have that $N_1 \cap \aleph_{\vec f^g(1)} = N_0 \cap \aleph_{\vec f^g(1)}$, since $\aleph_1 \subseteq N^*_{1,i}$ for $g(1) \leq i < g(2)$.  In particular, $\chi_{N_1}(1) = \chi_{N_0}(1)$.

We continue in this way, producing a decreasing sequence of elementary submodels $M \succ N_0 \succ N_1 \succ N_2 \succ \dots$ such that the intersections with the $\aleph_n$'s stabilizes:  For all $i>0$, $N_i \cap \aleph_{\vec f^g(i)} = N_{i-1} \cap \aleph_{\vec f^g(i)}$.  By following the same procedure as in the above paragraphs, we ensure that for $j>i$, $\chi_{N_i} (j) = f_0 \circ \dots \circ f_{g(i+1)-1}(j).$
Let $N_\omega = \bigcap_{i<\omega} N_i$.  Then $N_\omega$ is closed under $F$, $|N_\omega| = \aleph_\omega$, and for all $i$, $\chi_{N_\omega}(i) = \vec f^g(i)$.

For the claim about the function $s_1$, note that for any increasing $f$ such that $f(0) = 0$, we can define $g(n) = f(n)-n$, so that $f = s_1^g$.
\end{proof}

The above result produces a barrier to generalizing Silver's Theorem.  In broad strokes, the argument for Silver's Theorem works by taking a stationary set $S$ of elemenatary submodels of $H_{\aleph_\omega}$
and somehow finding a stationary $S' \subseteq S$ such that the function $M \mapsto \chi_M$ is constant on $S'$.  Of course in this situation, $\chi_M \in M$ for almost all $M \in S'$.  But this procedure cannot necessarily be carried out if the continuum is large enough:

\begin{corollary}
Suppose $\aleph_\omega$ is $f$-J\'onsson, and there is $m$ such that $m<f(m)$ and $\aleph_{m}\leq 2^{\aleph_0}$.  Then 
$\{ M \prec H_{\aleph_\omega} : \chi_M \notin M \}$
is stationary.
\end{corollary}

\begin{proof}
Let $F : \aleph_\omega^{<\omega} \to \aleph_\omega$ and let $M \prec (H_{\aleph_\omega},\in,F)$ be such that $|M|= \aleph_\omega$ and $\chi_M = f$.  If $f \notin M$, we are done.  Otherwise, let $\pi : M \to \bar M$ be the transitive collapse.  By hypothesis, there is $m$ such that $\pi(\aleph_m) < \aleph_m \leq 2^{\aleph_0}$, so there is $r \in 2^\omega \setminus M$.  Recursively define a monotone function $g$ by $g(0) = 0$, and $g(i) = g(i-1)+r(i)$ for $i >0$.  Let $N^* \prec \bar M$ be such that $f \in N^*$ and $\chi_{N^*} = f^g$.  Let $N = \pi^{-1}[N^*]$.  Then $\chi_N = f \circ f^g$.  For each $n$, $f \circ f^g (n) = f^{g(n)+1}(n)$.  For $n \geq m$, $f(n)>n$, and from $f$ and $f \circ f^g$ we can compute the unique value $i$ such that $f^i(n) = f \circ f^g(n)$.  Thus we can compute the values of $g$ above $m$, and from this recover the real number $r$.  Since $r \notin M \supseteq N$, $\chi_N \notin N$.  Note that $N$ is closed under $F$.
\end{proof}

If $f \in \omega^\omega$ is increasing, let 
$$[f] = \{ f^g : g \in \omega^\omega \text{ is monotone}\}.$$
Then $[f]$ is a closed subset of the Baire space.  This is because if $h \notin [f]$, then there must be some $n$ such that $h(n) \not= f^m(n)$ for any $m$.  Any function $h'$ such that $h'(n) = h(n)$ is also not in $[f]$, so there is a basic open set containing $h$ and disjoint from $[f]$.  Now if we choose continuum-many almost-disjoint subsets of $\omega$, $\{ X_\alpha : \alpha < 2^{\aleph_0} \}$, and let $f_\alpha$ be the increasing enumeration of $X_\alpha$, then for any $\alpha$ and any $h \in [f_\alpha]$, $\ran(h) \subseteq \ran(f_\alpha)$.  Thus $\{ [f_\alpha] : \alpha < 2^{\aleph_0} \}$ is a collection of pairwise disjoint closed sets.  The following question naturally arises:
\begin{question}
Suppose $\aleph_\omega$ is J\'onsson and let $\theta \geq \aleph_\omega$ be a cardinal.  What is the descriptive complexity of the set $\{ \chi_M : M \prec H_\theta, |M \cap \aleph_\omega| = \aleph_\omega \}$?
What can we say when $2^{\aleph_0} < \aleph_\omega$?
\end{question}

\bibliographystyle{amsplain.bst}
\bibliography{jonsson.bib}

\end{document}